\documentclass[preprint,12pt]{article}

\usepackage{amsmath,amsthm,amssymb}
\usepackage{hyperref}
\usepackage{dsfont}
\usepackage{enumerate}

\usepackage{marvosym}

\allowdisplaybreaks[4]

\newtheorem{lemma}{Lemma}[section]
\newtheorem{theorem}{Theorem}[section]
\newtheorem{corollary}{Corollary}[section]

\newtheorem{remark}{Remark}[section]

\numberwithin{equation}{section}

\DeclareMathOperator*{\osc}{osc}

\newcommand{\norm}[1]{\lVert#1\rVert}

\newcommand{\sref}[1]{Section~\ref{#1}}
\newcommand{\tref}[1]{\textsl{Theorem~\ref{#1}}}
\newcommand{\lref}[1]{\textsl{Lemma~\ref{#1}}}
\newcommand{\eref}[1]{\textsc{(\ref{#1})}}

\begin{document}
%%%%%%%%%%%%%%%%%%%%%%%%%%%%%%%%%%%%%%%%%%%%%%%%%

\title{Global $W^{2,\delta}$ estimates for a type of singular
fully nonlinear elliptic equations
\footnote{This research is supported by NSFC.11171266.}}
\author{Dongsheng Li\\
\href{mailto:lidsh@mail.xjtu.edu.cn}{lidsh@mail.xjtu.edu.cn}\\
{School of Mathematics and Statistics,}\\
{Xi'an Jiaotong University, Xi'an 710049, China}
\and Zhisu Li (\Letter)\footnote{Corresponding author.}\\
\href{mailto:lizhisu@stu.xjtu.edu.cn}{lizhisu@stu.xjtu.edu.cn}\\
{School of Mathematics and Statistics,}\\
{Xi'an Jiaotong University, Xi'an 710049, China}}

\date{}

\maketitle
%%%%%%%%%%%%%%%%%%%%%%%%%%%%%%%%%%%%%%%%%%%%%%%%%

\begin{abstract}
We obtain global $W^{2,\delta}$ estimates
for a type of singular fully nonlinear elliptic equations
where the right hand side term belongs to $L^\infty$.
The main idea of the proof is
to slide paraboloids from below and above to touch
the solution of the equation,
and then to estimate the low bound of
the measure of the set of contact points
by the measure of the set of vertex points.
\end{abstract}

%%%%%%%%%%%%%%%%%%%%%%%%%%%%%%%%%%%%%%%%%%%%%%%%%

\section{Introduction}

In this paper,
we obtain global $W^{2,\delta}$ estimates
for viscosity solutions of
the singular fully nonlinear elliptic inequalities
\begin{equation}\label{eq.ineqs}
|Du|^{-\gamma}\mathcal{P}_{\lambda,\Lambda}^-(D^2u)-|Du|^{1-\gamma}\leq f
\leq |Du|^{-\gamma}\mathcal{P}_{\lambda,\Lambda}^+(D^2u)+|Du|^{1-\gamma}
~~\mbox{in}~B_1,
\end{equation}
where $B_1$ is the unit open ball of $\mathds{R}^n$,
$\mathcal{P}_{\lambda,\Lambda}^\pm$ are the Pucci extremal operators,
$0<\lambda\leq\Lambda<\infty$, $0\leq\gamma<1$
and the right hand side term $f\in C^0\cap L^{\infty}(B_1)$.

%%%%%
The class of solutions of inequalities \eref{eq.ineqs}
includes solutions of several kinds of important equations.
The most common equation is
the singular fully nonlinear elliptic equation
in the following type
\[|Du|^{-\gamma}F(D^2u,Du,u,x)=f~~\mbox{in}~B_1,\]
where $0\leq\gamma<1$,
$F(0,0,\cdot,\cdot)\equiv0$
and $F$ is uniformly elliptic
(see \cite{CC} and \cite{Win}), that is,
\[\mathcal{P}_{\lambda,\Lambda}^-(M-N)-b|p-q|
\leq F(M,p,r,x)-F(N,q,s,x)
\leq \mathcal{P}_{\lambda,\Lambda}^+(M-N)+b|p-q|\]
for all $M,N\in S(n)$, $p,q\in\mathds{R}^n$,
$r,s\in\mathds{R}$ and $x\in B_1$,
where $0<\lambda\leq\Lambda<\infty$ and $b\geq0$.
The investigation of equations of this type
has made much progress in recent years.
I. Birindelli and F. Demengel proved comparison principle \cite{BD04}
and $C^{1,\alpha}$ estimate \cite{BD10}.
G. D\'{a}vila, P. Felmer and A. Quaas proved
Alexandroff-Bakelman-Pucci (ABP for short) estimate \cite{DFQ09}
and Harnack inequality \cite{DFQ10}.
To the best of our knowledge,
$W^{2,\delta}$ estimate for this kind of equation
is only known for $\gamma=0$,
that is the uniformly fully nonlinear elliptic equation.
%%%%%
In 1986, F.-H. Lin \cite{Lin} first established the interior
$W^{2,\delta}$ estimates for uniformly elliptic equations of
non-divergent type with measurable coefficients, with the help of
Fabes-Stroock type reverse H{\"o}lder inequality,
estimates of Green's function
and the ABP estimates.
%%%%%
In 1989, L. A. Caffarelli \cite{Caf}
applied ABP estimate, Calder\'{o}n-Zygmund cube decomposition technique,
barrier function method and touching by tangent paraboloid method
to obtain interior $W^{2,\delta}$ estimates
for viscosity solutions of
\[\mathcal{P}_{\lambda,\Lambda}^-(D^2u)\leq f
\leq \mathcal{P}_{\lambda,\Lambda}^+(D^2u),\]
and then he use such $W^{2,\delta}$ estimates
to get interior $W^{2,p}$ estimates for solutions of
\[F(D^2u,x)=f(x),\]
where the oscillation of $F(M,x)$ in $x$ is small
and the homogeneous equations with constants coefficients:
$F(D^2v(x),x_0)=0$ have $C^{1,1}$ interior estimates
(or $F$ is concave) (see also \cite{CC}).
%%%%%
$W^{2,p}$ estimates up to the boundary
were proved by N. Winter \cite{Win} in 2009.
%%%%%
Another famous example of singular equation which satisfies \eref{eq.ineqs}
is the singular $p$-Laplace equation
\begin{equation}\label{eq.pla}
\Delta_p u=f~~\mbox{in}~B_1,
\end{equation}
where $1<p\leq2$ and
\begin{eqnarray*}
\Delta_p u(x)&:=&\mbox{div}\left(|Du(x)|^{p-2}Du(x)\right)\\
&=&|Du(x)|^{-(2-p)}
\left(\delta_{ij}-(2-p)\cdot\frac{D_{i}u(x)D_{j}u(x)}{|Du(x)|^2}\right)D_{ij}u(x).
\end{eqnarray*}
In \cite{Tol}, P. Tolksdorf proved that
each $W^{1,p}\cap C^0(B_1)$-weak solution
of \eref{eq.pla} with $f\in L^{\infty}(B_1)$
is $W^{2,p}_{loc}\cap W^{1,p+2}_{loc}(B_1)$ (see also \cite{Lind}).
\\

%%%%%
We give an elementary proof of $W^{2,\delta}$ estimates
for viscosity solutions of
singular fully nonlinear elliptic
inequalities of the type \eref{eq.ineqs}.
Our estimates are global but the proof
does not need to flatten the boundary as usual, that is,
instead of separating it into interior estimates and boundary estimates,
we do it directly by using a new type of covering lemma.
The basic idea is to slide paraboloids
from below and above to touch
the solution of the equation, and then
to estimate the low bound of
the measure of the set of contact points
by the measure of the set of vertex points.
This idea has originated in the work of X. Cabr{\'e} \cite{Cab}
and continued in the work of O. Savin \cite{Sav} (see also \cite{IS}).
Following the same idea,
J.-P. Daniel \cite{Dan} proved an estimate equivalent to
local $W^{2,\delta}$ estimate for uniformly parabolic equation.
%%%%%
For singular fully nonlinear elliptic equations, intuitively,
once we have a universal control of $\norm{Du}_{L^\infty}$,
for instance $C^{1,\alpha}$ estimate (see \cite{BD10}),
the $W^{2,\delta}$ estimate will then be a natural corollary
of the traditional results of
\cite{Caf}, \cite{CC} and \cite{Win}.
But our method does not depend on
any \emph{a priori} estimate of $Du$
and it does not use maximum principles,
so we can deal with a large class of equations
as illustrated above.\\

%%%%%%%%%%%%%%%%%%%%%%%%%%%%%%%%%%%%%%%%%%%%%%%%%

The main result of this paper is the following theorem.
\begin{theorem}\label{th.w2d}
Let $0\leq\gamma<1$.
Assume that $u\in C^0(\overline{B_1})$ satisfies (\ref{eq.ineqs})
in the viscosity sense, with $f\in C^0\cap L^{\infty}(B_1)$.
Then $u\in W^{2,\delta}(B_1)$ and
\[\norm{u}_{W^{2,\delta}(B_1)}
\leq C\left(\norm{u}_{L^{\infty}(B_1)}
+\norm{f}_{L^{\infty}(B_1)}^{\frac{1}{1-\gamma}}\right)\]
for any $\delta\in(0,\sigma)$,
where $\sigma=\sigma(n,\lambda,\Lambda)>0$
and $C=C(n,\lambda,\Lambda,\delta)>0$.
\end{theorem}

\begin{remark}
The above global $W^{2,\delta}$ estimates in $B_1$
can be easily extended to those in some general
domain $\Omega\subset\mathds{R}^n$.
For example, $\Omega$ is bounded and can be decomposed
as the union of a collection of balls
with uniformly finite overlapping and
with uniformly lower bound radius.\\
\end{remark}

The paper is organized as follows.
We start in \sref{se.pre} by giving some notations and preliminary tools.
In \sref{se.gwe}, we first reduce \tref{th.w2d} to
\lref{le.w2dc} by rescaling and normalization,
and then the remainder of this section is devoted to
the proof of \lref{le.w2dc}.

\section{Preliminaries}\label{se.pre}

In this paper,
$B_r(x)$ denotes $\{y\in\mathds{R}^n:|y-x|<r\}$ and $B_r$ denotes $B_r(0)$.
$S(n)$ denotes the linear space of symmetric $n\times n$ real matrices.
$I$ denotes the identity matrix.
\\

Given two functions
$u$ and $v:\Omega\subset\mathds{R}^n\rightarrow\mathds{R}$
and a point $x_0\in\Omega$,
we say that
\emph{$u$ touches $v$ by below at $x_0$ in $\Omega$}
and denote it briefly by \emph{$u\overset{x_0}\eqslantless v$ in $\Omega$}, if
$u(x_0)=v(x_0)$ and $u(x)\leq v(x)$, $\forall x\in\Omega$.
\\

For a given continuous function $u:U\subset\mathds{R}^n\rightarrow\mathds{R}$,
we slide the concave paraboloid of opening $\kappa>0$ and vertex $y$
\[-\frac{\kappa}{2}|x-y|^2+C\]
from below in $U$ (by increasing or decreasing $C$)
till it touch the graph of $u$ for the first time.
If the contact point is $x_0$,
we then have
\[C=u(x_0)+\frac{\kappa}{2}|x_0-y|^2
=\underset{x\in U}\inf\left(u(x)+\frac{\kappa}{2}|x-y|^2\right).\]

For a given closed set $V\subset\mathds{R}^n$
and a continuous function $u:B_1\rightarrow\mathds{R}$,
we introduce the definitions of \emph{the contact sets} as follows:
\begin{eqnarray*}
T_\kappa^{-}(V)&:=&T_\kappa^{-}(u,V)
:=\bigg\{x_0\in B_1\mid\exists y\in V~\mbox{such~that}~\\
&~&~~u(x_0)+\frac{\kappa}{2}|x_0-y|^2
=\underset{x\in B_1}\inf\left(u(x)+\frac{\kappa}{2}|x-y|^2\right)\bigg\}\\
&=&~\bigg\{x_0\in B_1\mid\exists y\in V~\mbox{such~that}\\
&~&~~-\frac{\kappa}{2}|x-y|^2+u(x_0)+\frac{\kappa}{2}|x_0-y|^2
~\overset{x_0}\eqslantless u~\mbox{in}~B_1\bigg\},
\end{eqnarray*}
%%%%%
\[T_\kappa^+(V):=T_\kappa^+(u,V):=T_\kappa^-(-u,V)\]
and
\[T_\kappa(V):=T_\kappa(u,V):=T_\kappa^-(u,V)\cap T_\kappa^+(u,V).\]

For simplicity of notation, we will write $T_\kappa^\pm$
instead of $T_\kappa^\pm(u,\overline{B_1})$ when there is no confusion.
Note that $T_\kappa^\pm$ are closed in $B_1$.

We remark that the contact set $T^-_\kappa(u,V)$
has the dual functionality
of $\{u=\Gamma_u\}$ and $\underline{G}_M(u,\Omega)$ in \cite{CC} simultaneously.
The former digs up information on the equation,
while the later measures the second derivatives of the solution.
\\

Given $0<\lambda\leq\Lambda$, we define the
so called \emph{maximal and minimal Pucci extremal operators}
$\mathcal{P}^+_{\lambda,\Lambda}$
and $\mathcal{P}^-_{\lambda,\Lambda}$ (see also \cite{CC}) as follows
\[\mathcal{P}^+_{\lambda,\Lambda}(X)
:=\lambda\sum_{e_i(X)<0}e_i(X)+\Lambda\sum_{e_i(X)>0}e_i(X)\]
and
\[\mathcal{P}^-_{\lambda,\Lambda}(X)
:=\lambda\sum_{e_i(X)>0}e_i(X)+\Lambda\sum_{e_i(X)<0}e_i(X),\]
where $X\in S(n)$ and $e_i(X)$ denote the eigenvalues of $X$.
We will always abbreviate $\mathcal{P}^{\pm}_{\lambda,\Lambda}(X)$
to $\mathcal{P}^{\pm}(X)$.

For convenience,
we state some basic properties of the Pucci extremal operators as below:
\begin{enumerate}
\item
$\mathcal{P}^\pm(kX)=k\mathcal{P}^\pm(X)$,
$\mathcal{P}^\pm(-X)=-\mathcal{P}^\mp(X)$,
$\forall X\in S(n)$, $\forall k\geq0$.
\item
$\mathcal{P}^-(X)+\mathcal{P}^-(Y)\leq \mathcal{P}^-(X+Y)
\leq \mathcal{P}^-(X)+\mathcal{P}^+(Y)\leq \mathcal{P}^+(X+Y)$
$\leq \mathcal{P}^+(X)+\mathcal{P}^+(Y)$, $\forall X,Y\in S(n)$.
\item
$\mathcal{P}^+_{\lambda,\Lambda}(X)=\Lambda\mathrm{tr}X$
and $\mathcal{P}^-_{\lambda,\Lambda}(X)=\lambda\mathrm{tr}X$,
provided $X\geq0$, $X\in S(n)$.\\
\end{enumerate}

We recall the definition of viscosity solutions
(see \cite{CC} for more details).
For example, we say that $u\in C^0(B_1)$ \emph{satisfies}
\[
F(D^2u,Du,u,x)\leq f~~~\mbox{in}~~B_1
\]
\emph{in the viscosity sense},
if $\forall\varphi\in C^2(B_1)$, $\forall x_0\in B_1$,
\[
\varphi\overset{x_0}\eqslantless u~\mbox{in}~U(x_0)
\Rightarrow
F\left(D^2\varphi(x_0),D\varphi(x_0),\varphi(x_0),x_0\right)\leq f(x_0),
\]
where $U(x_0)$ is an open neighborhood of $x_0$ in $B_1$.
\\

We need the following equivalent descriptions
of $L^p$-integrability.
\begin{lemma}\label{le.lp} (see Lemma 7.3. in \cite{CC})
Let g be a nonnegative and measurable function in a
bounded domain $\Omega\subset\mathds{R}^n$.
Suppose that $\eta>0$, $M>1$ and $0<p<\infty$.
Then
\[g\in L^p(\Omega)~\Leftrightarrow~s
:=\sum_{k=1}^{\infty}M^{pk}|\{x\in\Omega\mid g(x)>\eta M^k\}|<\infty\]
and
\[C^{-1}s\leq\|g\|_{L^p(\Omega)}^p\leq C(s+|\Omega|),\]
where $C>0$ is a constant depending only on $\eta$, $M$ and $p$.\\
\end{lemma}

In the last part of this section, we introduce the following consequence
of Vitali's covering lemma, which has similar functionality to
the Calder{\'o}n-Zygmund cube decomposition lemma (see \cite{CC})
but has the advantage of giving global estimates directly.
This result is slightly different from that (growing ink-spots lemma) in \cite{IS},
but the idea of the proof is similar, which according to \cite{IS},
was first introduced by Krylov.
\begin{lemma}\label{le.cl}
Let $0<\mu<1$.
Assume that $E\subset F$ are closed subsets of $B_1$
and $E\neq\emptyset$.
Suppose that for any open ball $B\subset B_1$, if $B\cap E\neq\emptyset$,
then $|B\cap F|\geq\mu|B|$.
Then $|B_1\setminus F|\leq (1-\mu/{5^n})|B_1\setminus E|$.
\end{lemma}

\begin{proof}
It suffices to prove that $|F\setminus E|\geq\frac{\mu}{5^n}|B_1\setminus E|$.

For any $x\in B_1\setminus E$, by the openness of $B_1\setminus E$,
there exist open balls contained in $B_1\setminus E$ and containing $x$,
we choose one of the largest of them and denote it by $B^x$.

We claim that $|B^x\cap F|\geq\mu|B^x|$. Otherwise, since $E\neq\emptyset$,
and hence $B^x\subset B_1\setminus E\subsetneqq B_1$,
we may enlarge $B^x$ a little bit, denoted by $\widetilde{B}^x$,
such that $B^x\subset\widetilde{B}^x\subset B_1$
and $|\widetilde{B}^x\cap F|<\mu|\widetilde{B}^x|$.
By the hypothesis of the lemma, $\widetilde{B}^x\cap E=\emptyset$,
thus $\widetilde{B}^x\subset B_1\setminus E$,
which contradicts the definition of $B^x$.
Furthermore, since $B^x\cap F\setminus E=B^x\cap F$,
it follows that $|B^x\cap F\setminus E|\geq\mu|B^x|$.

Now consider the covering
$\underset{x\in B_1\setminus E}\cup B^x \supset B_1\setminus E$.
By the Vitali covering lemma,
there exists an at most countable set of points $x_i\in B_1\setminus E$,
such that $\{B^{x_i}\}_{i}$ are disjoint
and $\underset{i}\cup 5B^{x_i}\supset B_1\setminus E$.
Hence we have
\begin{eqnarray*}
\left|F\setminus E\right|
&\geq& \left|\left(\underset{i}\cup B^{x_i}\right)\cap F\setminus E\right|
=\left|\underset{i}\cup (B^{x_i}\cap F\setminus E)\right|
=\underset{i}\sum \left|B^{x_i}\cap F\setminus E\right|\\
&\geq& \mu\underset{i}\sum\left|B^{x_i}\right|
=\frac{\mu}{5^n}\underset{i}\sum\left|5B^{x_i}\right|
\geq\frac{\mu}{5^n}\left|\underset{i}\cup 5B^{x_i}\right|
\geq\frac{\mu}{5^n}\left|B_1\setminus E\right|.
\end{eqnarray*}
\end{proof}

\section{Global $W^{2,\delta}$ estimates}\label{se.gwe}
We first give the following lemma.
\begin{lemma}\label{le.w2dc}
Let $0\leq\gamma<1$.
Assume that $u\in C^0(\overline{B_1})$ satisfies (\ref{eq.ineqs})
in the viscosity sense, where $f\in C^0\cap L^{\infty}(B_1)$.
Then there exists
$\sigma=\sigma(n,\lambda,\Lambda)>0$
such that for any $\delta\in(0,\sigma)$,
if $\norm{u}_{L^{\infty}(B_1)}\leq 1/16$
and $\norm{f}_{L^{\infty}(B_1)}\leq 1$,
then
\[\norm{u}_{W^{2,\delta}(B_1)}\leq C(n,\lambda,\Lambda,\delta).\]
\end{lemma}

To prove \tref{th.w2d}, it suffices to prove \lref{le.w2dc}.
Indeed, suppose $u$ satisfies the hypothesis of \tref{th.w2d}.
Let
\[a:=\left(16\norm{u}_{L^\infty(B_1)}
+\norm{f}_{L^\infty(B_1)}^{\frac{1}{1-\gamma}}+\varepsilon\right)^{-1},\]
where $\varepsilon>0$. Then the scaled function $\widetilde{u}(x):=au(x)$ solves
\[|D\widetilde{u}|^{-\gamma}\mathcal{P}_{\lambda,\Lambda}^-(D^2\widetilde{u})
-|D\widetilde{u}|^{1-\gamma}
\leq a^{1-\gamma}f=:\widetilde{f}
\leq |D\widetilde{u}|^{-\gamma}\mathcal{P}_{\lambda,\Lambda}^+(D^2\widetilde{u})
+|D\widetilde{u}|^{1-\gamma}
~~\mbox{in}~B_1,\]
and satisfies $\norm{\widetilde{u}}_{L^\infty(B_1)}\leq1/16$
and $\norm{\widetilde{f}}_{L^{\infty}(B_1)}\leq1$.
Therefore, if
\[\norm{\widetilde{u}}_{W^{2,\delta}(B_1)}\leq C(n,\lambda,\Lambda,\delta),\]
by scaling back to $u$ and letting $\varepsilon\rightarrow0$, we obtain
\[
\norm{u}_{W^{2,\delta}(B_1)}\leq Ca^{-1}
\leq C\left(\|u\|_{L^{\infty}(B_1)}+\|f\|_{L^{\infty}(B_1)}^{\frac{1}{1-\gamma}}\right),
\]
which is the assertion of \tref{th.w2d}.\\

We establish \lref{le.w2dc} by the following two lemmas.
First we need the density lemma (\lref{le.dl})
which is a key lemma in this paper
and the strategy of its proof is modified from that in \cite{Sav}.

\begin{lemma}\label{le.dl}Let $0\leq\gamma<1$ and $K\geq1$.
Assume that $u\in C^0(\overline{B_1})$ satisfies
\begin{equation}\label{eq.dl}
|Du|^{-\gamma}\mathcal{P}_{\lambda,\Lambda}^-(D^2u)-|Du|^{1-\gamma}
\leq f~~\mbox{in}~B_1
\end{equation}
in the viscosity sense, where $f\in C^0\cap L^{\infty}(B_1)$
and $\norm{f}_{L^{\infty}(B_1)}\leq1$.
Then there exist constants
$M=M(n,\lambda,\Lambda)>1$ and $0<\mu=\mu(n,\lambda,\Lambda)<1$, such that
if $B_r(x_0)\subset B_1$ satisfies
\[B_r(x_0)\cap T^-_K\neq\emptyset,\]
then
\[\left|B_r(x_0)\cap T^-_{KM}\right|\geq \mu|B_r(x_0)|.\]
\end{lemma}

\begin{proof}
By assumption, there exist
$x_1\in B_r(x_0)\cap T^-_K$
and $y_1\in \overline{B_1}$ such that
\[P^-_{K,y_1}(x)
:=-\frac{K}{2}|x-y_1|^2+u(x_1)
+\frac{K}{2}|x_1-y_1|^2\overset{x_1}\eqslantless u~\mbox{in}~B_1.\]
The proof now will be divided into three steps.

\emph{Step 1.}
We prove that there exist $x_2\in B_{r/2}(x_0)$ and $C_0=C_0(n,\lambda,\Lambda)>0$
such that
\begin{equation}\label{eq.c1}
u(x_2)-P^-_{K,y_1}(x_2)\leq C_0Kr^2.
\end{equation}

Set
\[\psi(x):=P^-_{K,y_1}(x)+Kr^2\phi\left(\frac{|x-x_0|}{r}\right),\]
where $\phi(t)=e^Ae^{-At^2}-1$ with $A=A(n,\lambda,\Lambda)>1$ to be determined later.
Let $x_2\in \overline{B_r(x_0)}$ such that
\[(u-\psi)(x_2)=\underset{\overline{B_r(x_0)}}\min (u-\psi).\]
Since
$(u-\psi)|_{\partial B_r(x_0)}\geq0$
and
\[(u-\psi)(x_2)\leq(u-\psi)(x_1)=-Kr^2\phi\left(\frac{|x_1-x_0|}{r}\right)<0,\]
we deduce that
$x_2\in B_r(x_0)$
and \[u(x_2)<\psi(x_2)=P^-_{K,y_1}(x_2)+Kr^2\phi\left(\frac{|x_2-x_0|}{r}\right)
\leq P^-_{K,y_1}(x_2)+e^AKr^2.\]
Hence, by letting $C_0:=e^A$, we now only need to show that
there exists an $A=A(n,\lambda,\Lambda)>1$
such that $x_2\in B_{r/2}(x_0)$.

To obtain a contradiction,
suppose that $x_2$ is outside of $B_{r/2}(x_0)$
for any $A>1$.
Let $t:=\frac{|x_2-x_0|}{r}$. We have $1/2\leq t<1$.
Since
\[\psi+\underset{\overline{B_r(x_0)}}\min(u-\psi)
\overset{x_2}\eqslantless u~~\mbox{in}~B_r(x_0),\]
by applying the definition of the viscosity solution of \eref{eq.dl},
we have
\begin{equation}\label{eq.xA}
\mathcal{P}_{\lambda,\Lambda}^-(D^2\psi(x_2))-|D\psi(x_2)|
\leq |D\psi(x_2)|^{\gamma}f(x_2).
\end{equation}
%%%%%
Since
\[\left|DP^-_{K,y_1}(x_2)\right|\leq CK,\]
\[\left|D^2P^-_{K,y_1}(x_2)\right|\leq CK,\]
\[\left|\frac{\phi_{t}}{t}\right|\leq CAe^Ae^{-At^2}\]
and
\[\left|\phi_{tt}\right|\geq \frac{1}{C}A^2e^Ae^{-At^2},\]
we see that
\[|D\psi(x_2)|^{\gamma}\leq C\left(KAe^A e^{-At^2}\right)^{\gamma}\]
and
\begin{eqnarray*}
&~&\mathcal{P}_{\lambda,\Lambda}^-(D^2\psi)(x_2)-|D\psi(x_2)|\\
&\geq& CK\left(|\phi_{tt}|-1\right)
-CK\left(\left|\frac{\phi_{t}}{t}\right|+1\right)-CKAe^A e^{-At^2}\\
&\geq& CKAe^A e^{-At^2},
\end{eqnarray*}
where $A$ is large enough
and all the constants $C$ depend only on $n,\lambda$ and $\Lambda$.
By \eref{eq.xA}, we conclude that
\[\left(KAe^A e^{-At^2}\right)^{1-\gamma}\leq 1,\]
which is impossible.

\emph{Step 2.}
We prove that there exists
$M=M(n,\lambda,\Lambda)>1$ such that
\begin{equation}\label{eq.TKM}
T^-_{KM}(V)
\subset B_r(x_0)\cap T^-_{KM},
\end{equation}
where
\[V:=\overline{B_{\frac{r(M-1)}{8M}}
\left(\frac{1}{M}y_1+\frac{M-1}{M}x_2\right)}.\]

$\forall\widetilde{x}\in T^-_{KM}(V)$,
$\exists\widetilde{y}\in V$
such that
\begin{equation*}
P^-_{KM,\widetilde{y}}(x)
:=-\frac{KM}{2}|x-\widetilde{y}|^2+u(\widetilde{x})
+\frac{KM}{2}|\widetilde{x}-\widetilde{y}|^2
\overset{\widetilde{x}}\eqslantless u~~\mbox{in}~B_1.
\end{equation*}
Since
\[
P^-_{KM,\widetilde{y}}(x)-P^-_{K,y_1}(x)
=-\frac{K(M-1)}{2}|x-y|^2+R,
\]
where
\begin{equation}\label{eq.ydef}
y:=\frac{M\widetilde{y}-y_1}{M-1}
\end{equation}
and $R=R\left(\widetilde{y},K,M,y_1,\widetilde{x},
u(\widetilde{x}),x_1,u(x_1)\right)$
both do not depend on $x$,
we have
\[P^-_{K,y_1}(x)-\frac{K(M-1)}{2}|x-y|^2+R
\overset{\widetilde{x}}\eqslantless u(x),~
\forall~x\in B_1.\]
%%%%%
Since $x_2\in B_{r/2}(x_0)\subset B_1$
and $y\in \overline{B_{r/8}(x_2)}$,
we see from \eref{eq.c1} that
\[R\leq u(x_2)-P^-_{K,y_1}(x_2)+\frac{K(M-1)}{2}|x_2-y|^2
\leq\left(C_0+\frac{M-1}{128}\right)Kr^2.\]
%%%%%
On the other hand,
\[0\leq u(\widetilde{x})-P^-_{K,y_1}(\widetilde{x})
=P^-_{KM,\widetilde{y}}(\widetilde{x})-P^-_{K,y_1}(\widetilde{x})
=-\frac{K(M-1)}{2}|\widetilde{x}-y|^2+R.
\]
%%%%%
It follows that
\[|\widetilde{x}-y|^2
\leq\frac{2}{M-1}\left(C_0+\frac{M-1}{128}\right)r^2
=\left(\frac{2C_0}{M-1}+\frac{1}{64}\right)r^2.\]
%%%%%
Let $M>\frac{128C_0}{3}+1$,
we obtain $|\widetilde{x}-y|<r/4$.
Thus
\[|\widetilde{x}-x_2|
\leq|\widetilde{x}-y|+|y-x_2|
<r/4+r/8=3r/8,\]
%%%%%
and hence
\begin{equation}\label{eq.tcb}
T^-_{KM}(V)\subset B_{3r/8}(x_2)
\subset\subset B_{r/2}(x_2)\subset B_r(x_0)\subset B_1.
\end{equation}

By \eref{eq.ydef}, we see that
$\forall \widetilde{y}\in V$,
$\exists y\in \overline{B_{r/8}(x_2)}$
such that
\[\widetilde{y}=\frac{1}{M}y_1+\frac{M-1}{M}y.\]
Since $y_1\in \overline{B_1}$
and
$\overline{B_{r/8}(x_2)}\subset
B_{5r/8}(x_0)\subset\subset B_{r}(x_0)\subset B_1$,
we obtain
\begin{equation}\label{eq.vcb}
V\subset\subset B_1,
\end{equation}
by the convexity of $B_1$.
Thus
\[T^-_{KM}(V)\subset T^-_{KM}(\overline{B_1})=T^-_{KM},\]
and hence
\begin{equation*}
T^-_{KM}(V)\subset B_r(x_0)\cap T^-_{KM}.
\end{equation*}

\emph{Step 3.}
We claim that
\begin{equation}\label{eq.VCT}
|V|\leq C\left|T^-_{KM}(V)\right|,
\end{equation}
where $C=C(n,\lambda,\Lambda)>0$.
If we prove this,
then by \eref{eq.TKM},
we will obtain
\begin{eqnarray*}
\left|B_r(x_0)\cap T^-_{KM}\right|
&\geq&\left|T^-_{KM}(V)\right|
\geq\frac{1}{C}|V|\\
&=&\frac{1}{C}\left(\frac{M-1}{8M}\right)^n\left|B_{r}\right|\\
&=:&\mu\left|B_r(x_0)\right|,
\end{eqnarray*}
which proves the lemma.

Hence it remains to prove \eref{eq.VCT}.
To do this, we need to regularize $u$
by the standard $\epsilon$-envelope method of Jensen (see \cite{CC}).
That is, for $\epsilon>0$, let
\[u_{\epsilon}(x):=\underset{z\in B_1}\inf\left(u(z)
+\frac{1}{\epsilon^4}|z-x|^2\right),~~\forall x\in B_1.\]
It is easy to see that
$u_\epsilon$ is $C^{1,1}$ a.e. in $B_1$,
$u_\epsilon\in C^0(B_1)$ and
$u_\epsilon\rightarrow u$ ($\epsilon\rightarrow 0+$)
uniformly on compact subsets of $B_1$.
Furthermore, we show that there exists $\epsilon_0>0$ sufficiently small
such that for any $\epsilon\in(0,\epsilon_0)$,
\begin{equation}\label{eq.te}
T^-_{KM}\left(u_\epsilon,V\right)
\subset B_{1-\epsilon}
\end{equation}
and $u_\epsilon$ satisfies
\begin{equation}\label{eq.ue}
|Du_\epsilon|^{-\gamma}\mathcal{P}_{\lambda,\Lambda}^-(D^2u_\epsilon)
-|Du_\epsilon|^{1-\gamma} \leq f_\epsilon~~\mbox{in}~B_{1-\epsilon}
\end{equation}
in the viscosity sense,
with $f_\epsilon$ to be given later which satisfies
$f_\epsilon\in C^0(B_{1-\epsilon})$,
$\norm{f_\epsilon}_{L^{\infty}(B_{1-\epsilon})}\leq 1$
and $f_\epsilon\rightarrow f$ ($\epsilon\rightarrow 0+$)
uniformly on compact subsets of $B_{1-\epsilon}$.
To see \eref{eq.te}, we only need to note
\eref{eq.tcb} and \eref{eq.vcb}.
We now verify \eref{eq.ue}.
Suppose that $\varphi\in C^2(B_1)$, $x_*\in B_{1-\epsilon}$
and
\[\varphi\overset{x_*}\eqslantless u_\epsilon~\mbox{in}~U(x_*),\]
where $U(x_*)$ is an open neighborhood of $x_*$ in $B_{1-\epsilon}$.
Let $\widetilde{x_*}\in \overline{B_1}$ such that
\[u_{\epsilon}(x_*)=\underset{z\in B_1}\inf\left(u(z)
+\frac{1}{\epsilon^4}|z-x_*|^2\right)
=u(\widetilde{x_*})+\frac{1}{\epsilon^4}|\widetilde{x_*}-x_*|^2.\]
In view of
\[|\widetilde{x_*}-x_*|^2
=\epsilon^4\left(u_\epsilon(x_*)-u(\widetilde{x_*})\right)
\leq\epsilon^4\cdot\underset{\overline{B_1}}\osc u,\]
we have $\widetilde{x_*}\in B_1$ provided $\epsilon$ is small enough.
Since
\begin{eqnarray*}
\varphi(x-\widetilde{x_*}+x_*)
&\leq & u_\epsilon(x-\widetilde{x_*}+x_*)\\
&=& \underset{z\in B_1}\inf\left(u(z)
+\frac{1}{\epsilon^4}|z-x+\widetilde{x_*}-x_*|^2\right)\\
&\leq& u(x)+\frac{1}{\epsilon^4}|\widetilde{x_*}-x_*|^2\\
&=& u(x)+u_\epsilon(x_*)-u(\widetilde{x_*})\\
&=& u(x)+\varphi(x_*)-u(\widetilde{x_*}),
\end{eqnarray*}
provided $x-\widetilde{x_*}+x_*\in U(x_*)$,
we deduce that
\[\varphi(x-\widetilde{x_*}+x_*)-\varphi(x_*)+u(\widetilde{x_*})
\overset{\widetilde{x_*}}\eqslantless u(x)
~~\mbox{in}~U(\widetilde{x_*}),\]
where $U(\widetilde{x_*})$ is
a small open neighborhood of $\widetilde{x_*}$ in $B_1$.
Hence
\[
|D\varphi(x_*)|^{-\gamma}
\mathcal{P}_{\lambda,\Lambda}^-(D^2\varphi(x_*))
-|D\varphi(x_*)|^{1-\gamma}
\leq f(\widetilde{x_*})=:f_\epsilon(x_*),
\]
which gives \eref{eq.ue}.

Now since $u_\epsilon$ is a.e. $C^{1,1}$,
we see that for almost every
$x\in T^-_{KM}\left(u_\epsilon,V\right)$,
there exists a unique
$y\in V$
which satisfies
\[Du_\epsilon(x)=-KM(x-y)\]
and
\[D^2u_\epsilon(x)\geq-KMI.\]
Hence we may define the mapping $y:x\mapsto y$,
$T^-_{KM}\left(u_\epsilon, V\right)
\rightarrow V$,
given by
\[y=x+\frac{1}{KM}Du_\epsilon(x).\]
Since
\[|Du_\epsilon(x)|=KM|x-y|\leq2KM\]
and
\[D_xy=I+\frac{1}{KM}D^2u_\epsilon(x)\geq0,\]
we conclude from \eref{eq.te} and \eref{eq.ue} that
\begin{eqnarray*}
\lambda\textrm{tr}(D_xy)
&=&\mathcal{P}_{\lambda,\Lambda}^-(D_xy)\\
&\leq&\mathcal{P}_{\lambda,\Lambda}^+(I)
+\mathcal{P}_{\lambda,\Lambda}^-\left(\frac{D^2u_\epsilon(x)}{KM}\right)\\
&\leq& n\Lambda+\frac{1}{KM}\left(|Du_\epsilon(x)|^{\gamma}|f_\epsilon(x)|
+|Du_\epsilon(x)|\right)\\
&\leq& n\Lambda+\frac{2}{(KM)^{1-\gamma}}+2\\
&\leq& n\Lambda+4 \leq C
\end{eqnarray*}
and
\[0\leq\det(D_xy)\leq\left(\frac{\textrm{tr}(D_xy)}{n}\right)^n\leq C,\]
%%%%%
where the constants $C$ depend only on $n,\lambda$ and $\Lambda$.
Hence, we have
\[|V|\leq\underset{T^-_{KM}\left(u_\epsilon, V\right)}\int\det(D_xy)dx
\leq C\left|T^-_{KM}\left(u_\epsilon, V\right)\right|,\]
by the area formula.

Finally, it is easy to check that
\[\overline{\underset{\epsilon\rightarrow0+}\lim}
T^-_{KM}\left(u_\epsilon, V\right)
\subset T^-_{KM}\left(u, V\right),\]
where
\[\overline{\underset{\epsilon\rightarrow0+}\lim}
T^-_{KM}\left(u_\epsilon, V\right)
:=\underset{0<\epsilon<\epsilon_0}\cap\underset{0<\delta<\epsilon}\cup
T^-_{KM}\left(u_\delta, V\right).\]
Hence we have
\begin{eqnarray*}
|V|&\leq& C\overline{\underset{\epsilon\rightarrow0+}\lim}
\left|T^-_{KM}\left(u_\epsilon, V\right)\right|\\
&\leq& C\left|\overline{\underset{\epsilon\rightarrow0+}\lim}
T^-_{KM}\left(u_\epsilon, V\right)\right|\\
&\leq& C\left|T^-_{KM}\left(u, V\right)\right|,
\end{eqnarray*}
which gives \eref{eq.VCT} and completes the proof.
\end{proof}
~

We now give a lemma which states
the decay of $\left|B_1\setminus T^-_{t}\right|$ in $t$.
\begin{lemma}\label{le.ed}
Let $0\leq\gamma<1$ and $K\geq1$.
Assume that $u\in C^0(\overline{B_1})$ satisfies
\begin{equation*}
|Du|^{-\gamma}\mathcal{P}_{\lambda,\Lambda}^-(D^2u)-|Du|^{1-\gamma}
\leq f~~\mbox{in}~B_1
\end{equation*}
in the viscosity sense,
where $f\in C^0\cap L^{\infty}(B_1)$,
$\norm{f}_{L^{\infty}(B_1)}\leq 1$
and
$\norm{u}_{L^{\infty}(B_1)}\leq 1/16$.
Then there exist constants
$M=M(n,\lambda,\Lambda)>1$
and $\theta=\theta(n,\lambda,\Lambda)\in(0,1)$
such that
\[\left|B_1\setminus T^-_{KM}\right|
\leq\theta\left|B_1\setminus T^-_{K}\right|.\]
\end{lemma}

\begin{proof}
We first show that
\begin{equation}\label{eq.ne}
B_1\cap T^-_K\neq\emptyset.
\end{equation}

Let $\widetilde{x}\in \overline{B_1}$
such that
\[u(\widetilde{x})+\frac{K}{2}|\widetilde{x}|^2
=\underset{x\in\overline{B_1}}\min\left(u(x)+\frac{K}{2}|x|^2\right)=:m.\]
We have
$-\frac{K}{2}|x|^2+m
\overset{\widetilde{x}}\eqslantless u$ in $\overline{B_1}$.
In particular, we obtain
$m\leq u(0)$
and
$-\frac{K}{2}|\widetilde{x}|^2+m=u(\widetilde{x})$.
Subtracting one from the other, we conclude that
\[\frac{K}{2}|\widetilde{x}|^2
\leq u(0)-u(\widetilde{x})
\leq\underset{B_1}{\osc}~u\leq2\norm{u}_{L^{\infty}(B_1)}\leq1/8,\]
which implies $|\widetilde{x}|\leq1/2$.
Thus $\widetilde{x}\in B_1\cap T^-_K$.

In view of \eref{eq.ne} and \lref{le.dl},
by applying \lref{le.cl} to $E:=T^-_K$ and $F:=T^-_{KM}$,
we obtain
\[\left|B_1\setminus T^-_{KM}\right|
\leq(1-\frac{\mu}{5^n})\left|B_1\setminus T^-_K\right|
=:\theta\left|B_1\setminus T^-_K\right|.\]
This completes the proof of \lref{le.ed}.
\end{proof}

\begin{corollary}\label{co.ed}
Under the assumptions of \lref{le.ed}, we have
\[\left|B_1\setminus T^-_{M^k}\right|
\leq |B_1|\theta^k,~\forall k\geq1.\]
\end{corollary}

\begin{proof}
Applying \lref{le.ed} to $K=M^{k-1},M^{k-2},...,M$ and $1$, we deduce that
\[\left|B_1\setminus T^-_{M^k}\right|
\leq\theta\left|B_1\setminus T^-_{M^{k-1}}\right|
\leq...\leq\theta^k\left|B_1\setminus T^-_1\right|
\leq|B_1|\theta^k.\]
\end{proof}

\lref{le.w2dc} now follows easily from \textsl{Corollary \ref{co.ed}}.
\\

\noindent\emph{Proof of \lref{le.w2dc}}.
(1) By definition,
$T^+_\kappa(u,\overline{B_1})=T^-_\kappa(-u,\overline{B_1})$.
Applying \textsl{Corollary \ref{co.ed}} to $-u$, we have
$\left|B_1\setminus T^+_{M^k}\right|\leq |B_1|\theta^k$, $\forall k\geq1$.
Therefore
\begin{eqnarray*}
\left|B_1\setminus T_{M^k}\right|
&=&\left|B_1\setminus \left(T^-_{M^k}\cap T^+_{M^k}\right)\right|\\
&\leq&\left|B_1\setminus T^-_{M^k}\right|
+\left|B_1\setminus T^+_{M^k}\right|\\
&\leq& 2|B_1|\theta^k,~\forall k\geq1.
\end{eqnarray*}
It follows that
\begin{equation}\label{eq.tt}
\left|B_1\setminus T_{t}\right|\leq Ct^{-\sigma},~\forall t>0,
\end{equation}
where $\sigma:=-\log_M\theta$ and $C:=2|B_1|/\theta$.

(2) To get $\norm{u}_{W^{2,\delta}(B_1)}\leq C$,
by the interpolation theorem (see \textsl{Theorem 7.28} of \cite{GT}),
we only need to show that
$\norm{D^2u}_{L^\delta(B_1)}\leq C$.
By \eref{eq.tt}, this can be obtained directly from \lref{le.lp}
(cf. \cite{CC}, pp. 6-7, 62-63, 66-67).

For heuristic purposes, we here additionally give a simple proof
for the special case that $u\in C^2(B_1)$.
Since
\[B_1\cap T_t\subset\{x\in B_1:-tI\leq D^2u(x)\leq tI\}
\subset\{x\in B_1:|D^2u(x)|\leq \sqrt{n}t\},\]
we see that
\[\{x\in B_1:|D^2u(x)|>\sqrt{n}t\}\subset B_1\setminus T_t.\]
Hence
\[|\{x\in B_1:|D^2u(x)|>\sqrt{n}t\}|
\leq|B_1\setminus T_t|\leq Ct^{-\sigma}.\]
Applying \lref{le.lp}, we obtain
\begin{eqnarray*}
\norm{D^2u}^\delta_{L^\delta(B_1)}
&\leq& C(n,M,\delta)\left(|B_1|
+\underset{k}\sum M^{\delta k}|\{x\in B_1:|D^2u(x)|>\sqrt{n}M^k\}|\right)\\
&\leq& C(n,M,\delta)\left(|B_1|+C\underset{k}\sum M^{(\delta-\sigma)k}\right)
\leq C(n,\lambda,\Lambda,\delta)
<\infty.
\end{eqnarray*}
\qed

%%%%%%%%%%%%%%%%%%%%%%%%%%%%%%%%%%%%%%%%%%%%%%%%%

\begin{thebibliography}{CCCC}

\bibitem[BD04]{BD04}
Birindelli, I., Demengel, F.:
Comparison principle and Liouville type results for
singular fully nonlinear operators.
Ann. Fac. Sci. Toulouse Math. \textbf{13}, 261--287 (2004)

\bibitem[BD10]{BD10}
Birindelli, I., Demengel, F.:
Regularity and uniqueness of the first eigenfunction
for singular fully nonlinear operators.
J. Differ. Equations. \textbf{249}, 1089--1110 (2010)

\bibitem[Cab97]{Cab}
Cabr\'{e}, X.:
Nondivergent elliptic equations on manifolds with nonnegative curvature.
Commun. Pur. Appl. Math. \textbf{50}(7), 623--665 (1997)

\bibitem[Caf89]{Caf}
Caffarelli, L.A.:
Interior a priori estimates for solutions of fully nonlinear equations.
Ann. Math. \textbf{130}, 189--213 (1989)

\bibitem[CC95]{CC}
Caffarelli, L.A., Cabr\'{e}, X.:
Fully nonlinear elliptic equations.
Colloquium Publications 43,
American Mathematical Society, Providence, RI (1995)

\bibitem[Dan15]{Dan}
Daniel, J.-P.:
Quadratic expansions and partial regularity for fully nonlinear
uniformly parabolic equations.
Calc. Var. Partial Dif. \textbf{54}(1), 183--216 (2015)

\bibitem[DFQ09]{DFQ09}
D\'{a}vila, G., Felmer, P., Quaas, A.:
Alexandroff-Bakelman-Pucci estimate for singular or
degenerate fully nonlinear elliptic equations.
Comptes Rendus Mathematique \textbf{347}, 1165--1168 (2009)

\bibitem[DFQ10]{DFQ10}
D\'{a}vila, G., Felmer, P., Quaas, A.:
Harnack inequality for singular fully nonlinear operators and some existence results.
Calc. Var. Partial Dif. \textbf{39}, 557--578 (2010)

\bibitem[GT98]{GT}
Gilbarg, D., Trudinger, N.S.:
Elliptic Partial Differential Equations of Second Order (3rd edition).
Springer-Verlag, Berlin (1998)

\bibitem[IS13]{IS}
Imbert, C., Silvestre, L.:
Estimates on elliptic equations that hold only where the gradient is large.
J. Eur. Math. Soc. (in press).
\href{http://arxiv.org/abs/1306.2429}{arXiv:1306.2429}

\bibitem[Lin86]{Lin}
Lin, F.-H.:
Second derivative $L^p$-estimates for elliptic equations of nondivergent type.
P. Am. Math. Soc. \textbf{96}(3), 447--451 (1986)

\bibitem[Lind05]{Lind}
Lindqvist, P.:
Notes on the p-Laplace equation (2005)
\url{http://www.math.ntnu.no/~lqvist/p-laplace.pdf}

\bibitem[Sav07]{Sav}
Savin, O.:
Small perturbation solutions for elliptic equations.
Commun. Part. Diff. Eq. \textbf{32}, 557--578 (2007)

\bibitem[Tol84]{Tol}
Tolksdorf, P.:
Regularity for a more general class of
quasilinear elliptic equations.
J. Differ. Equations \textbf{51}, 126--150 (1984)

\bibitem[Win09]{Win}
Winter, N.:
$W^{2,p}$ and $W^{1,p}$ estimates at the boundary
for solutions of fully nonlinear uniformly elliptic equations.
Z. Anal. Anwend. \textbf{28}, 129--164 (2009)

\end{thebibliography}
\end{document}